\documentclass[reqno]{amsart}

\usepackage{amssymb}
\usepackage{mathrsfs}
\usepackage{amsmath}
\usepackage{amsthm}
\usepackage{url}
\usepackage{stackrel} 
\usepackage{courier}
\usepackage[all,cmtip]{xy} 
\usepackage{multicol} 
\usepackage{verbatim} %  useful for \begin{comment} \end{comment}  
\usepackage{mathtools}  
\usepackage{hyperref}
\usepackage{multicol}  
\usepackage{enumerate} 
\usepackage{bookmark} 
\usepackage{ytableau}
\usepackage{microtype}
 
\catcode`~=11 % make LaTeX treat tilde (~) like a normal character
\newcommand{\urltilde}{\kern -.15em\lower .7ex\hbox{~}\kern .04em}  
\catcode`~=13 % revert back to treating tilde (~) as an active character

\newcommand{\lie}{\mathop{\mathrm{Lie}}\nolimits}

\newcommand{\I}{\mathop{\mathrm{I}_n}\nolimits}

\makeatletter
\newif\if@borderstar
   \def\bordermatrix{\@ifnextchar*{%
       \@borderstartrue\@bordermatrix@i}{\@borderstarfalse\@bordermatrix@i*}%
   }
   \def\@bordermatrix@i*{\@ifnextchar[{\@bordermatrix@ii}{\@bordermatrix@ii[()]}}
   \def\@bordermatrix@ii[#1]#2{%
   \begingroup
     \m@th\@tempdima8.75\p@\setbox\z@\vbox{%
       \def\cr{\crcr\noalign{\kern 2\p@\global\let\cr\endline }}%
       \ialign {$##$\hfil\kern 2\p@\kern\@tempdima & \thinspace %
       \hfil $##$\hfil && \quad\hfil $##$\hfil\crcr\omit\strut %
       \hfil\crcr\noalign{\kern -\baselineskip}#2\crcr\omit %
       \strut\cr}}%
     \setbox\tw@\vbox{\unvcopy\z@\global\setbox\@ne\lastbox}%
     \setbox\tw@\hbox{\unhbox\@ne\unskip\global\setbox\@ne\lastbox}%
     \setbox\tw@\hbox{%
       $\kern\wd\@ne\kern -\@tempdima\left\@firstoftwo#1%
         \if@borderstar\kern2pt\else\kern -\wd\@ne\fi%
       \global\setbox\@ne\vbox{\box\@ne\if@borderstar\else\kern 2\p@\fi}%
       \vcenter{\if@borderstar\else\kern -\ht\@ne\fi%
         \unvbox\z@\kern-\if@borderstar2\fi\baselineskip}%
         \if@borderstar\kern-2\@tempdima\kern2\p@\else\,\fi\right\@secondoftwo#1 $%
     }\null \;\vbox{\kern\ht\@ne\box\tw@}%
   \endgroup
   }
\makeatother

\newtheorem{theorem}{Theorem}[section] 

\newtheorem{lemma}[theorem]{Lemma}

\newtheorem{example}[theorem]{Example}
\newtheorem{definition}[theorem]{Definition}
\newtheorem{remark}[theorem]{Remark}

\begin{document}

\title[Semi-invariants of filtered quiver representations with at most two pathways]{Semi-invariants of filtered quiver representations with at most two pathways}  
\author{Mee Seong Im} 
\address{Department of Mathematics, University of Illinois at Urbana-Champaign, Urbana, IL 61801 USA}
\email{mim2@illinois.edu}  
\date{\today}

\begin{abstract}   
A pathway from one vertex of a quiver to another is a reduced path.   
We modify the classical definition of quiver representations and we prove that semi-invariant polynomials for filtered quiver representations come from diagonal entries if and only if the quiver has at most two pathways between any two vertices. 
Such class of quivers includes finite $ADE$-Dynkin quivers,   
affine $\widetilde{A}\widetilde{D}\widetilde{E}$-Dynkin quivers, star-shaped and comet-shaped quivers. 
Next, we explicitly write all semi-invariant generators for filtered quiver representations for framed quivers with at most two pathways between any two vertices; this result may be used to study constructions analogous to Nakajima's affine quotient and quiver varieties, which are, in special cases, 
$\mathfrak{M}_0^{F^{\bullet}}(n,1) := \mu_B^{-1}(0)/\!\!/B$ and 
$\mathfrak{M}^{F^{\bullet}}(n,1) := \mu_B^{-1}(0)^s/B$, respectively, where 
$\mu_B:T^*(\mathfrak{b}\times \mathbb{C}^n)\rightarrow \mathfrak{b}^*\cong \mathfrak{gl}_n^*/\mathfrak{u}$, $B$ is the set of invertible upper triangular $n\times n$ complex matrices, $\mathfrak{b}=\lie(B)$, and $\mathfrak{u}\subseteq \mathfrak{b}$ is the biggest unipotent subalgebra. 
\end{abstract}   

\maketitle

\bibliographystyle{amsalpha}  
 
\setcounter{tocdepth}{1}

\section{Introduction}\label{section:introduction}  
As separating a $G$-space into invariant (or weight) spaces or constructing polynomials invariant under a group action is a fundamental and important procedure in mathematics, 
Schofield-van den Bergh (\cite{MR1908144}), 
Derksen-Weyman (\cite{MR1758750}), 
and Domokos-Zubkov (\cite{MR1825166}) 
are a few mathematicians who have explicitly given strategies on producing (semi)-invariant polynomials for all quiver representations. 
%More intuitively, 
%the trace of an oriented cycle of a quiver as well as the trace of a path that begin and end at a framed vertex is an $\mathbb{G}_{\beta}=\prod_{i\in Q_0} GL_{\beta_i}(\mathbb{C})$-invariant function (\cite{MR958897}, \cite{MR1834739}, \cite{MR1623674}), while Derksen-Weyman in \cite{MR1758750} 
%give strategies for acyclic quivers. 

In this paper, we modify the classical construction of quiver representations in such a way that our construction is related to the Grothendieck-Springer resolution 
(\cite{MR2838836}, \cite{MR1649626}, \cite{Im-doctoral-thesis}, \cite{Nevins-GSresolutions}), 
KLR-algebras (\cite{MR2525917}, \cite{MR2763732}, \cite{Rouquier-2-Kac-Moody-algebras}, \cite{Stroppel-Webster-quiver-schur-algebras-q-fock-space}), 
universal quiver Grassmannians and universal quiver flag varieties (\cite{MR2772068}),   
and Lusztig's upper half $U^+$ of the universal enveloping algebra of a Kac-Moody algebra  
(\cite{MR1035415}, \cite{MR1182165}, \cite{MR1758244}), but such details will not be elaborated here.  
Instead, we describe the modification of quiver representations, which is as follows.  
Consider a sequence of vector spaces over each vertex of a quiver and restrict to the subspace of quiver representations that preserve this fixed sequence of vector spaces; we call such space {\em filtered quiver representations}. 
Furthermore, there is a unique largest unipotent subgroup of the set of complex invertible matrices over each vertex which preserves the filtration of vector spaces and acts as a change-of-basis.    
Thus, considering the product of these unipotent groups acting on the filtered quiver representation space, we give an explicit description of the ring of invariant polynomials.

Let $Q=(Q_0,Q_1)$ be a finite, connected, nonframed quiver and $\beta\in \mathbb{Z}_{\geq 0}^{Q_0}$ be a dimension vector.   
Assume  
$F^{\bullet}$ is a %the complete standard    
filtration of vector spaces at each vertex and   
$F^{\bullet}Rep(Q,\beta)$ is a subspace of   
$Rep(Q,\beta)$ whose representations preserve $F^{\bullet}$; 
the product $\mathbb{U}_{\beta}$ of largest unipotent subgroups of a parabolic group %the standard Borel % 
acts on  
$F^{\bullet}Rep(Q,\beta)$ as a change-of-basis.   
% Recall that a quiver is a Dynkin quiver if the underlying graph has the structure of a Dynkin graph (cf. Section~\ref{section:background}).  
%    
%  
% 
%\begin{assumption}\label{assumption:framed-quiver}  
A framed quiver $Q^{\dagger}=(Q_0^{\dagger},Q_1^{\dagger})$ is obtained from $Q$  
by adding a single new vertex $i'$ to $Q_0$ 
together with an arrow from $i'$ to $i$.     
Let $\beta^{\dagger}=(\beta_1,\ldots, \beta_{Q_0},m)\in \mathbb{Z}_{\geq 0}^{Q_0^{\dagger}}$   
be the dimension vector for the framed quiver. We define 
\[ 
F^{\bullet} Rep(Q^{\dagger},\beta^{\dagger})  
:=  F^{\bullet} Rep(Q,\beta)\oplus M_{\beta_{i}\times m},  
\] 
where $M_{\beta_{i}\times m}$ is the space of all $\beta_i\times m$ complex matrices. 
%\end{assumption}  

Throughout this paper, assume $\beta=(n,\ldots,n)$ and $F^{\bullet}$  is the complete standard filtration of vector spaces at each (nonframed) vertex of $Q$. 
Let $\mathfrak{t}_n$ be the set of complex diagonal matrices in the set $\mathfrak{gl}_n$ of $n\times n$ complex matrices.  

The definition of a pathway between two vertices is given in Definition~\ref{definition:pathway-of-quiver}. 

\begin{theorem}\label{theorem:two-paths-max-quiver-semi-invariants}    
%Let $Q$ be a quiver and let $\beta=(n,\ldots, n)$, a dimension vector.    
%Let $F^{\bullet}$ be a complete standard filtration of vector spaces over each vertex and let   
%$\mathbb{U}_{\beta}  
%$ be a product of the maximal unipotent subgroup of the standard Borel $B$. 
%Then  
$Q$ is a quiver with at most two distinct pathways between any two vertices if and only if  $\mathbb{C}[F^{\bullet}Rep(Q,\beta)]^{\mathbb{U}_{\beta}}\cong \mathbb{C}[\mathfrak{t}^{\oplus Q_1}]$. 
\end{theorem}   

\begin{remark}
Quivers satisfying 
% Studying parabolic invariants on parabolic Lie algebras is a classical problem which hasn't been systematically studied. 
Theorem~\ref{theorem:two-paths-max-quiver-semi-invariants} include  
$ADE$-Dynkin quivers, affine $\widetilde{A} \widetilde{D} \widetilde{E}$-quivers, 
star-shaped and comet-shaped quivers (\cite{Im-doctoral-thesis}).   
\end{remark}

Theorem~\ref{theorem:two-paths-max-quiver-semi-invariants} has a number of important consequences,   
including if $\mathbb{U}_{\beta}$-invariants for filtered quiver representations   
only come from diagonal blocks (i.e., the semisimple part), then $Q$   
has at most two pathways between any two vertices.   
Furthermore,   
this implies that Domokos-Zubkov's technique is applicable to the filtered quiver representation space if the quiver has at most two pathways between any two vertices.   
It is, in fact, shown in \cite{Im-doctoral-thesis}     
that off-diagonal entries contribute as invariant polynomials if $Q$ has more than two pathways between some of its vertices. We note that Theorem~\ref{theorem:two-paths-max-quiver-semi-invariants} could easily be generalized if we take the filtration of vector spaces at each (nonframed) vertex $i\in Q_0$ to be 
$G_i^{\bullet}:\mathbb{C}^0\subseteq \mathbb{C}^{\gamma_1}\subseteq \mathbb{C}^{\gamma_2}\subseteq \ldots \subseteq \mathbb{C}^{\gamma_k}=\mathbb{C}^n$; in such setting, only the Levi subalgebra components in the filtered representation space contribute $\mathbb{U}_{\beta}$-invariant polynomials.

\begin{theorem}\label{theorem:two-paths-max-quiver-semi-invariants-framed}  
Let  $Q^{\dagger}$
be a framed quiver with at most two distinct pathways between any two vertices. 
Let 
$\beta=(n,\ldots,n,m)\in \mathbb{Z}^{Q_0^{\dagger}}$ be a dimension vector, where $m$ is associated to the framed vertex.   
% Let $F^{\bullet}$ be the complete standard filtration of vector spaces at nonframed vertices and let 
% $\mathbb{U}_{\beta}:=U^{Q_0-1}$ be a product of unipotent subgroups.  
Then 
$\mathbb{C}[F^{\bullet}Rep(Q^{\dagger},\beta^{\dagger})]^{\mathbb{U}_{\beta}}$ 
$\cong$ 
$\mathbb{C}[\mathfrak{t}^{\oplus Q_1 }]\otimes_{\mathbb{C}} \mathbb{C}[\{ f\}]$,  
where
\begin{equation}\label{eq:affine-dk-quiver-general-n-inv-poly}
  f = \sum_{\nu} g_{\nu}({}_{(\alpha)}a_{st})  
(J_{\nu}|I_{\nu})_{\mathbf{\Phi_{ }}\mathbf{\Psi_{ }} \cdots \mathbf{\Gamma_{ }}}, 
\end{equation}  
where   
$g_{\nu}({}_{(\alpha)}a_{st} )\in \mathbb{C}[\mathfrak{t}^{\oplus Q_1}]$  
 and $(J_{\nu}|I_{\nu})_{\mathbf{\Phi_{ }}\mathbf{\Psi_{ }} \cdots \mathbf{\Gamma_{ }}}$    
 is the block standard bideterminant given in Definition~\ref{definition:bideterminant-poly-for-theorem}. 
\end{theorem}

\subsection{Acknowledgment}      
The author would like to thank Thomas Nevins for thought-provoking discussions on research problems in representation theory and algebraic geometry and for numerous fruitful discussions.  
The author was supported by NSA grant H98230-12-1-0216, by Campus Research Board, and by NSF
grant DMS 08-38434.

\section{Background}\label{section:background}   
 
We refer to the author's doctoral thesis (\cite{Im-doctoral-thesis}), 
Crawley-Boevey (\cite{Crawley-Boevey-rep-quivers}) and Ginzburg (\cite{Ginzburg-Nakajima-quivers}) 
for background on quivers and their representations.

\begin{definition}\label{definition:reduced-path}   
Let $Q$ be a quiver.   
Let $p=a_k\cdots a_2 a_1$ be a path where $a_i\in Q_1$ are arrows.   
If $p$ is a cycle, then   
we define $p^m$ to be the path composed with itself $m$ times, i.e., 
\[ 
p^m := p\circ p \circ \cdots \circ p = \underbrace{(a_k\cdots a_2 a_1)\cdots (a_k\cdots a_2 a_1)}_{m} = (a_k\cdots a_2 a_1)^m. 
\]  
% A nontrivial path $p=a_k\cdots a_2 a_1$ 
A path $p$  
is {\em reduced} if $[p]\not=0$ in   
$\mathbb{C}Q/\langle q^2: q\in \mathbb{C}Q, l(q)\geq 1 \rangle$. 
% 
%whenever 
%$p$ can be rewritten as 
%$a_k \cdots (a_{i_l}\cdots a_{i_1})^m \cdots a_2 a_1$,  
%then $m=1$. 
\end{definition}

%Note that we include trivial paths of a quiver to be reduced. 

\begin{definition}\label{definition:pathway-of-quiver} 
 A {\em pathway} from vertex $i$ to vertex $j$ is % a trivial or 
 a reduced path 
 from $i$ to $j$.
 %, which includes the trivial path $e_i$ if $i=j$.  
 We define 
{\em pathways} of a quiver $Q$ to be the set of  
 %a subalgebra of $\mathbb{C}Q$ consisting of 
 all pathways from vertex $i$ to vertex $j$, where $i,j\in Q_0$.  
\end{definition}   
  
Note that pathways (of a quiver $Q$)   
include trivial paths and they   
form a finite set since $Q$ is a finite quiver.   
We will now give an example of Definition~\ref{definition:pathway-of-quiver}.  
 
\begin{example}   
Consider the $2$-Jordan quiver: \[   
\xymatrix@-1pc{    
\stackrel{1.}{\bullet} \ar@(ru,rd)^{a_1} \ar@(ld,lu)^{a_2}   
}  
\]   
Then $a_2^2 a_1$ is a path but not a pathway since it is not reduced.   
However, the path $a_2 a_1$ is a pathway.  
\end{example}

Next, we make a distinction between quiver varieties and quiver representations.   
When one refers to quiver varieties, one usually means Nakajima quiver varieties, i.e., the Hamiltonian reduction of a (double) framed quiver representation space twisted by a nontrivial character (\cite{MR1604167}).   
In this paper, we will only work with quiver representations.

\subsection{Filtered quiver representations}\label{subsection:filtered-quiver-reps}

We give the construction of filtered quiver representations in the general setting. 
Let $Q$ be a quiver and let $\beta=(\beta_1,\ldots, \beta_{Q_0})\in \mathbb{Z}_{\geq 0}^{Q_0}$, a dimension vector.  
Let 
$F^{\bullet}:0\subseteq \mathbb{C}^{\gamma^1}\subseteq \mathbb{C}^{\gamma^2}\subseteq \ldots \subseteq \mathbb{C}^{\beta}$ 
be a filtration of vector spaces such that the filtration 
$F_i^{\bullet}:0\subseteq \mathbb{C}^{\gamma^1_i}\subseteq \mathbb{C}^{\gamma^2_i}\subseteq 
\ldots \subseteq \mathbb{C}^{\beta_i}$ of vector spaces is fixed at vertex $i$ for each $i\in Q_0$. 
Let $Rep(Q,\beta)$ be the representation space in the classical sense (without the filtration of vector spaces imposed). 
Then $F^{\bullet}Rep(Q,\beta)$ is a subspace of $Rep(Q,\beta)$ whose linear maps preserve the filtration of vector spaces at every level.  
Let $U_i\subseteq GL_{\beta_i}(\mathbb{C})$ be the largest unipotent group preserving the filtration of vector spaces at vertex $i$. 
Then the product $\mathbb{U}_{\beta}:= \displaystyle{\prod_{i\in Q_0}U_i}$ of unipotent groups acts on 
$F^{\bullet}Rep(Q,\beta)$ as a change-of-basis.

We will say a representation (or a matrix) is {\em general} if it has indeterminates or zero in its entries.
% and 
%a parabolic Lie algebra $\mathfrak{p}=\lie(P)$ is in {\em standard form}  
%if a general parabolic matrix has indeterminates along its block diagonal and upper triangular portion of the matrix and zero % below the diagonal blocks.    

\section{Proof of Theorem~\ref{theorem:two-paths-max-quiver-semi-invariants}}\label{section:two-paths-max-quiver-semi-invariants}

\begin{proof}  
First, we will prove from right to left. Assume  
$\mathbb{C}[F^{\bullet}Rep(Q,\beta)]^{\mathbb{U}_{\beta}} \cong \mathbb{C}[\mathfrak{t}^{\oplus Q_1}]$.  
%
%If $Q$ has a framing at some vertex, say at $1$: $\xymatrix@-1pc{ \stackrel{1'}{\circ}\ar[rr]^{a_0} & & \stackrel{1}{\bullet}}$, 
%then consider a representation of vertices $1'$ and $1$ and the arrow $a_0$:   
%$\xymatrix@-1pc{\stackrel{\mathbb{C}^m}{\circ}\ar[rr]^{A_0} & & \stackrel{\mathbb{C}^n}{\bullet} }$   
%where $A_0\in M_{n\times m}$ is a general matrix. 
%Consider the maximal unipotent subgroup $U\subseteq GL_n(\mathbb{C})$ acting on $A_0$ 
%via left multiplication: 
%$u.A_0 =uA_0$. 
%By direct calculation, it is clear that for each $1\leq k\leq m$,  
%$(n,k)$-entry of $A_0$ is an invariant polynomial. 
%By assumption, invariant polynomials only come from general representation of arrows whose tail and head are nonframed vertices. This implies that $Q$ cannot be a framed quiver.  
% 
% 
% 
We will prove that $Q$ has at most two distinct pathways between any two vertices. 
For a contradiction, suppose $Q$ has $3$ or more pathways at some vertices. 
There are two cases to consider: 
\begin{enumerate}[ {[}1{]} ]  
\item\label{item:3pathways-in-one-vertex} $Q$ has $3$ or more pathways from a vertex to itself, e.g., $\hspace{6mm}\xymatrix@-1pc{ \ar@(lu,ld) \bullet \ar@(ru,rd) }$, 
\item\label{item:3pathways-from-one-vertex-to-another} $Q$ has $3$ or more pathways from vertex $i$ to vertex $j$, where $i\not=j$, e.g., $\xymatrix@-1pc{ \stackrel{i}{\bullet} \ar@/^/[rr] \ar[rr] \ar@/_/[rr] & & \stackrel{j}{\bullet} }$
or $\xymatrix@-1pc{ \stackrel{i}{\bullet} \ar@/^/[rr] \ar@/_/[rr] & & \stackrel{j}{\bullet} \ar@(ru,rd)}\hspace{6mm}$
or $\xymatrix@-1pc{ \stackrel{i}{\bullet} \ar@/^/[rr] \ar@/_/[rr] & & \bullet \ar@/^/[rr] \ar@/_/[rr] & & \stackrel{j}{\bullet}  }$ 
or $\xymatrix@-1pc{ \stackrel{i}{\bullet} \ar[rr] & & \bullet \ar@(rd,ld) \ar@/^/[rr] \ar@/_/[rr] & & \stackrel{j}{\bullet}  }$.  
\end{enumerate}    

\bigskip 

Consider Case~[\ref{item:3pathways-in-one-vertex}].   
Without loss of generality,   
relabel the vertex with $3$ or more pathways as $1$ and label the distinct pathways from vertex $1$ to itself as   
\[  
a_1 = p_{i_1} \cdots p_{i_{\alpha}}, \hspace{4mm}  
a_2 = p_{j_1} \cdots p_{j_{\beta}}, \hspace{2mm} \ldots,     
a_m = p_{k_1} \cdots p_{k_{\gamma}},  
\]    
where $m\geq 2$ and each $a_l$ is not the trivial path.   
 Write a general representation of $a_1,\ldots, a_m$ as   
 $A_1,\ldots, A_m$, each of which is in $\mathfrak{b}$ with polynomial entries.   
 Let 
 $A_1=({}_{(1)}a_{ij})$ and $A_2 = ({}_{(2)}a_{ij})$,     
 and consider the polynomial   
 \[ f(A_1,\ldots, A_m)=({}_{(1)}a_{11} - {}_{(1)}a_{22}){}_{(2)}a_{12}  
                    -({}_{(2)}a_{11} - {}_{(2)}a_{22}){}_{(1)}a_{12}. 
                    \]     
  For $u\in U\cong U\times \I^{Q_0-1}\subseteq \mathbb{U}_{\beta}$ where $u$ acts on the restricted tuple
  $(A_1,\ldots, A_m)$ of the filtered representation space via 
  $u.(A_1,\ldots, A_m)=(u A_1 u^{-1},\ldots, u A_m u^{-1} )$ and $\I$ is the $n\times n$ identity matrix,  
  the coordinates change as 
  ${}_{(l)}a_{12}\mapsto {}_{(l)}a_{12}+u_{12}( {}_{(l)}a_{22}-{}_{(l)}a_{11}  )$  
  under the group action.   
  A restricted tuple means we ignore and suppress all other components of $F^{\bullet}Rep(Q,\beta)$ for simplicity as they are not currently relevant. 
  So   
  \[   
  \begin{aligned}  
  u.&f(A_1,\ldots, A_m)   \\  
   &= ({}_{(1)}a_{11} - {}_{(1)}a_{22})({}_{(2)}a_{12}+u_{12}({}_{(2)} a_{22}-{}_{(2)}a_{11} ))
                    -({}_{(2)}a_{11} - {}_{(2)}a_{22})({}_{(1)}a_{12}+u_{12}({}_{(1)} a_{22}-{}_{(1)}a_{11} )) 
                    																										\\ 
 &= ({}_{(1)}a_{11} - {}_{(1)}a_{22}){}_{(2)}a_{12}   -({}_{(2)}a_{11} - {}_{(2)}a_{22}){}_{(1)}a_{12}  \\ 
                    &\hspace{4mm}+ u_{12}({}_{(1)}a_{11} - {}_{(1)}a_{22})({}_{(2)} a_{22}-{}_{(2)}a_{11})
                    -u_{12}({}_{(2)}a_{11} - {}_{(2)}a_{22})({}_{(1)} a_{22}-{}_{(1)}a_{11})  \\
                    &= f(A_1,\ldots, A_m). \\ 
  \end{aligned} 
  \] 
This implies $\mathbb{C}[F^{\bullet}Rep(Q,\beta)]^{\mathbb{U}_{\beta}} \supsetneq \mathbb{C}[\mathfrak{t}^{\oplus Q_1}]$, which is a contradiction.

Now consider Case~[\ref{item:3pathways-from-one-vertex-to-another}]. 
Suppose 
\[ 
a_1 = p_{i_1} \cdots p_{i_{\alpha}}, \hspace{4mm}  
a_2 = p_{j_1} \cdots p_{j_{\beta}}, \hspace{2mm} \ldots,  
a_k = p_{l_1} \cdots p_{l_{\gamma}}  
\] 
are distinct pathways from vertex $i$ to vertex $j$, where $i\not=j$, $k\geq 3$, and each $a_{l}$ is not the trivial path. 
Write a general representation of the pathways $a_1,\ldots, a_k$ as 
$A_1=({}_{(1)}a_{ij}),\ldots, A_k=({}_{(k)}a_{ij})\in \mathfrak{b}$. 
Consider the $\mathbb{U}_{\beta}$-action on $F^{\bullet}Rep(Q,\beta)$. 
In particular, consider $\I^{Q_0-2}\times U^2$ acting locally on 
$\mathfrak{b}^{\oplus k}$ via 
\[ 
(\I,\ldots, \I, u,v).(A_1,\ldots, A_k) = (u A_1v^{-1},\ldots, u A_k v^{-1} ). 
\] 
Consider the polynomial 
\[ 
\begin{aligned} 
g(A_1,\ldots, A_k) = ({}_{(1)} a_{11} {}_{(2)} a_{22} &- {}_{(1)} a_{22} {}_{(2)}a_{11} ) {}_{(3)} a_{12}   
+ ({}_{(3)} a_{11} {}_{(1)} a_{22} - {}_{(3)} a_{22} {}_{(1)}a_{11} ) {}_{(2)} a_{12}    \\  
&+ ({}_{(2)} a_{11} {}_{(3)} a_{22} - {}_{(2)} a_{22} {}_{(3)}a_{11} ) {}_{(1)} a_{12}.   \\ 
\end{aligned} 
\] 
Then 
\[ 
\begin{aligned} 
(\I,\ldots, \I,u,v).g(A_1,\ldots, A_k) 
&= ({}_{(1)} a_{11} {}_{(2)} a_{22} - {}_{(1)} a_{22} {}_{(2)}a_{11} ) ({}_{(3)} a_{12} -{}_{(3)}a_{22}u_{12}+{}_{(3)}a_{11}v_{12} )  \\ 
&\:\:+ ({}_{(3)} a_{11} {}_{(1)} a_{22} - {}_{(3)} a_{22} {}_{(1)}a_{11} ) ({}_{(2)} a_{12} -{}_{(2)}a_{22}u_{12}+{}_{(2)}a_{11}v_{12} )   \\  
&\:\:+ ({}_{(2)} a_{11} {}_{(3)} a_{22} - {}_{(2)} a_{22} {}_{(3)}a_{11} ) ({}_{(1)} a_{12} -{}_{(1)}a_{22}u_{12}+{}_{(1)}a_{11}v_{12} )   \\ 
&= g(A_1,\ldots, A_k) 
+  ({}_{(1)} a_{11} {}_{(2)} a_{22} - {}_{(1)} a_{22} {}_{(2)}a_{11} ) ( 
 -{}_{(3)}a_{22}u_{12}+{}_{(3)}a_{11}v_{12} )  \\ 
&\:\:+ ({}_{(3)} a_{11} {}_{(1)} a_{22} - {}_{(3)} a_{22} {}_{(1)}a_{11} ) (
 -{}_{(2)}a_{22}u_{12}+{}_{(2)}a_{11}v_{12} )   \\  
&\:\:+ ({}_{(2)} a_{11} {}_{(3)} a_{22} - {}_{(2)} a_{22} {}_{(3)}a_{11} ) ( 
 -{}_{(1)}a_{22}u_{12}+{}_{(1)}a_{11}v_{12} )  
 = g(A_1,\ldots, A_k). \\  
\end{aligned}
\] 
This implies that $g$ is an invariant polynomial, which contradicts that $\mathbb{C}[F^{\bullet}Rep(Q,\beta)]^{\mathbb{U}_{\beta}}
\cong \mathbb{C}[\mathfrak{t}^{\oplus Q_1}]$. 
Thus, $Q$ is a quiver with at most two pathways between any two vertices.

Now suppose $Q$ is a quiver with at most two distinct pathways between any two vertices.   
We first define a notion of total ordering on pairs of integers and then choose the least pair.    
We then list all possible local models of arrows of $Q$ at a fixed vertex.   
Writing an invariant polynomial which depends on this least pair,   
we carefully choose a subgroup of $\mathbb{U}_{\beta}$ and show that the invariant polynomial must only depend on diagonal coordinates of each general matrix in the filtered representation space.  
We will now give the full proof.

We label the arrows of $Q$ as $a_1$, $a_2$, $\ldots$, $a_{Q_1}$.   
Since it is clear that     
$\mathbb{C}[\mathfrak{t}^{\oplus Q_1}]\subseteq   \mathbb{C}[F^{\bullet}Rep(Q,\beta)]^{\mathbb{U}_{\beta}}$,  
we will prove the other inclusion.   
Consider a general representation of $F^{\bullet}Rep(Q,\beta)$, which is a tuple of matrices.  
We define a total ordering $\leq$ on pairs $(i,j)$, where $1\leq i\leq j\leq n$, by defining  
 $(i,j)\leq (i',j')$ if either   
 \begin{itemize}   
 \item $i < i'$ or   
 \item $i=i'$ and $j> j'$.  
 \end{itemize}   
 Let   
 $f\in \mathbb{C}[F^{\bullet}Rep(Q,\beta)]^{\mathbb{U}_{\beta}}$.   
  For each $(i,j)$, we can write  
 \begin{equation}\label{eq:aij-invariant-function-borel-at-most-two-pathways}
 f = \sum_{|K|\leq d} a_{ij}^K f_{ij,K}, \mbox{ where } f_{ij,K}\in \mathbb{C}[\{ {}_{(\alpha)}a_{st}: (s,t)\not=(i,j)\} ], 
 a_{ij}^K := \prod_{\alpha\in Q_1} {}_{(\alpha)}a_{ij}^{k_{\alpha}}, 
 \mbox{ and } 
 |K|=\sum_{\alpha=1}^{Q_1} k_{\alpha}.
 \end{equation}  
 Fix the least pair (under $\leq$) $(i,j)$ with $i< j$ for which there exists $K\not=(0,\ldots, 0)$ with $f_{ij,K}\not=0$; 
 we will continue to denote it by $(i,j)$. 
 If no such $(i,j)$ exists, then $f\in \mathbb{C}[{}_{(\alpha)}a_{ii}]$ and we are done. 
 Let $K=(k_1,\ldots, k_{Q_1})$. 
 Let $m\geq 1$ be the least integer satisfying the following: 
for all $p<m$, if some component $k_p$ in $K$ is strictly greater than $0$, then $f_{ij,K}=0$. 
%  
%  
% Let $m\geq 1$ be the least integer for which $f_{ij,K}=0$ whenever $K=(k_1,\ldots, k_{Q_1})$ and 
% some $k_p>0$, for $p<m$. 
 Relabel the head of the arrow corresponding to a general representation $A_m$ of arrow $a_m$ 
 as vertex $m$ (this is the same $m$ as in the previous sentence).  
%Note that such $m$ exists since the arrows of $Q$ are labelled as $a_1$, $a_2$, $\ldots$, $a_{Q_1}$. 

 Let $U_{ij}$ be the subgroup of matrices of the form 
 $u_{ij} := (\I,\ldots, \I, \widehat{u}_m,\I, \ldots, \I)$, where 
 $\widehat{u}_m$ is the matrix with $1$ along the diagonal, the variable $u$ in the $(i,j)$-entry, and $0$ elsewhere. 
 Let $u_{ij}^{-1}:=(\I,\ldots, \I, \widehat{u}_m^{-1},\I, \ldots, \I)$. 
 Then since $u_{ij}^{-1}$ acts on $f$ via 
 \[ 
 \begin{aligned}
 u_{ij}^{-1}.f(A_1,\ldots, A_{Q_1}) &= 
 f(u_{ij}.(A_1,\ldots, A_{Q_1}))  \\ 
 &= 
 \begin{cases}  
 f(A_1, 
  \ldots,  
  \widehat{u}_{m}  A_{m'}     , 
  \ldots,  
  A_{Q_1}) &\mbox{ whenever vertex $m$ is a sink of arrow $a_{m'}$}, \\ 
   f(A_1, 
  \ldots,  
    A_{m'} \widehat{u}_{m}^{-1}    , 
  \ldots,  
  A_{Q_1}) &\mbox{ whenever vertex $m$ is a source of arrow $a_{m'}$}, \\ 
  f(A_1,  
  \ldots,  
  \widehat{u}_{m}  A_{m'}  \widehat{u}_{m}^{-1}, 
  \ldots,   
  A_{Q_1}) &\mbox{ whenever arrow $a_{m'}$ is a loop at vertex $m$}, \\ 
 \end{cases}  \\ 
 \end{aligned}
 \] 
 \begin{comment}
 $u_{ij}^{-1}.f(A_1,\ldots, A_{Q_1})$ $=$ 
 $f(u_{ij}.(A_1,\ldots, A_{Q_1}))$, which is 
   $f(A_1, 
  \ldots,  
  \widehat{u}_{m}  A_m     , 
  \ldots,  
  A_{Q_1})$ if $m$ is a sink of arrow $a_m$ 
  or 
  %    $f(A_1, 
  %\ldots,  
  %  A_m\widehat{u}_{m}^{-1}, 
  %\ldots,  
  %A_{Q_1})$ 
  %if $m$ is a source to arrow $a_m$, or 
  $f(A_1,$ 
  $\ldots,$  
  $\widehat{u}_{m}  A_m  \widehat{u}_{m}^{-1},$ 
  $\ldots,$  
  $A_{Q_1})$  
  if $a_m$ is a loop at vertex $m$,  
  \end{comment} 
 \begin{equation}\label{eq:ar-dynkin-Uij-action} 
  u_{ij}^{-1}.{}_{(\alpha)}a_{st}= 
 \begin{cases} 
 {}_{(m')}a_{ij} +{}_{(m')}a_{jj}u &\mbox{ if } \alpha=m', (s,t)=(i,j), \mbox{ and } m \mbox{ is a sink to } a_{m'}, \\ 
%%%%%  {}_{(m)} a_{it}+{}_{(m)}a_{jt}  u & \mbox{ if } \alpha=m, s=i, t>j, \mbox{ and } m \mbox{ is a sink to } a_m,\\ 
 {}_{(m')}a_{ij} -{}_{(m')}a_{ii}u &\mbox{ if } \alpha=m', (s,t)=(i,j), \mbox{ and } m \mbox{ is a source to } a_{m'}, \\
% {}_{(m)}a_{sj} - {}_{(m)}a_{si} u & \mbox{ if } \alpha=m, s<i, t=j,   \mbox{ and } m \mbox{ is a source to } a_m,\\  
 {}_{(m')}a_{ij} +({}_{(m')}a_{jj}- {}_{(m')}a_{ii})u &\mbox{ if } \alpha=m', (s,t)=(i,j), \mbox{ and } a_{m'} \mbox{ is a loop at } m, \\ 
%%%%%  {}_{(m)}a_{sj} - {}_{(m)}a_{si}u &\mbox{ if } \alpha=m, s<i, t=j,  \mbox{ and } a_m \mbox{ is a loop at } m, \\ 
%%%%%  {}_{(m)}a_{it} + {}_{(m)}a_{jt}u &\mbox{ if } \alpha=m, s=i, t>j,  \mbox{ and } a_m \mbox{ is a loop at } m, \\  
 {}_{(\alpha)}a_{st} &\mbox{ if }s>i \mbox{ or }s=i \mbox{ and }t<j.  \\ 
 \end{cases}  
 \end{equation}

% Since arguments if $m$ is a source is similar to if $m$ is a sink, we will assume that $m$ is not a source.  
%
Locally at vertex $m$, $Q$ has one of the following local models:  
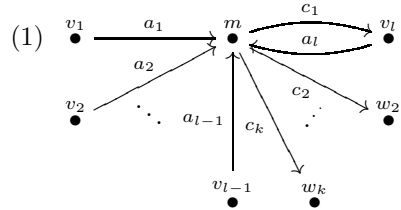
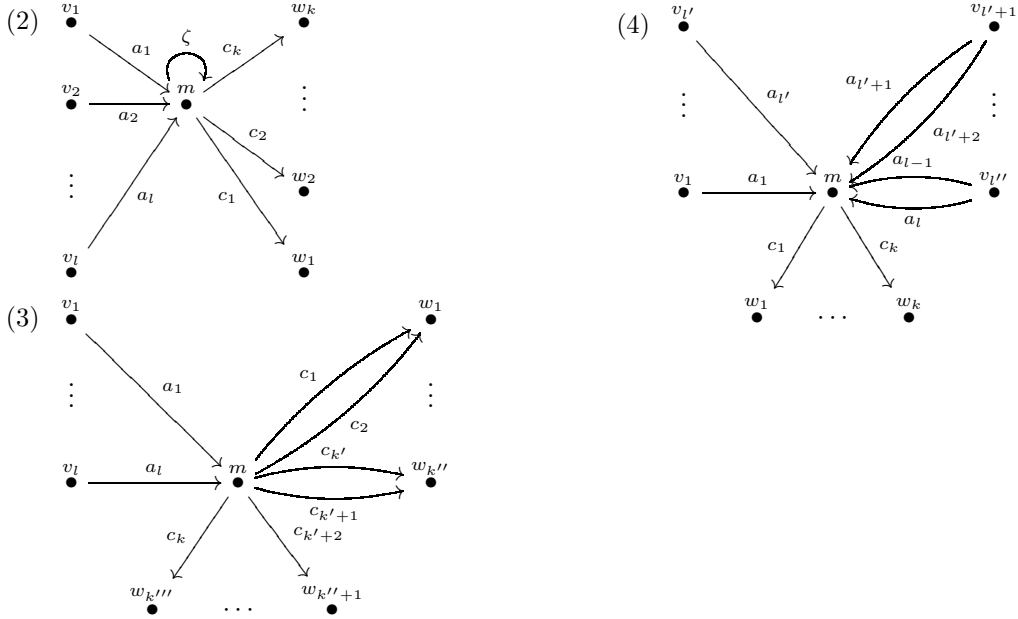
\begin{multicols}{2}  
 \begin{enumerate}  
 \item\label{item:twopathways-case7}  $\xymatrix@-1pc{
  \stackrel{v_1}{\bullet} \ar[rr]^{a_1}  & & \stackrel{m}{\bullet} \ar[ddr]_{c_k} \ar[drr]_{c_{2}} \ar@/^/[rr]^{c_1} & & \ar@/^/[ll]_{a_{l}}  \stackrel{v_l}{\bullet}  \\
  \stackrel{v_2}{\bullet}  \ar[rru]^{a_2}  & \ddots & & .^{.^.}   & \stackrel{w_{2}}{\bullet}\\ 
 &  &\stackrel{v_{l-1}}{\bullet} \ar[uu]^{a_{l-1}} &\stackrel{w_k}{\bullet} & \\
 }$
% % Case 8 [below]
% \item\label{item:twopathways-case8}     $\xymatrix@-1pc{ 
% \stackrel{v_1}{\bullet} \ar[rr]^{a_1} & & \stackrel{m}{\bullet} \ar@(lu,ru)^{a_m} & & %\stackrel{v_l}{\bullet}\ar[ll]_{a_l} \\ 
% & & & & \\
% & \stackrel{v_2}{\bullet} \ar[uur]^{a_2} & \cdots & \stackrel{v_{l-1}}{\bullet} \ar[luu]_{a_{l-1}} & \\  
% }$  
% % Case 9 [below]  
% \item\label{item:twopathways-case9}  $\xymatrix@-1pc{ 
% \stackrel{v_1}{\bullet} & & \ar[ll]_{a_1} \ar[ddl]_{a_2} \stackrel{m}{\bullet} \ar@(lu,ru)^{a_m}  \ar[ddr]^{a_{l-1}}  %\ar[rr]^{a_l} & & \stackrel{v_l}{\bullet} \\ 
% & & & & \\
% & \stackrel{v_2}{\bullet} & \cdots & \stackrel{v_{l-1}}{\bullet} & \\  
% }$  
%%  
 \item\label{item:twopathways-case10}   $\xymatrix@-1pc{ 
 \stackrel{v_1}{\bullet} \ar[drr]^{a_1}   & & & & \stackrel{w_{k}}{\bullet} \\ 
 \stackrel{v_2}{\bullet} \ar[rr]_{a_2}  & & \stackrel{m}{\bullet} \ar@(lu,ru)^{\zeta} \ar[drr]^{c_{2}} \ar[ddrr]_{c_{1}} \ar[urr]^{c_{k}} & & \vdots \\ 
 \vdots & & &  & \stackrel{w_{2}}{\bullet} \\  
 \stackrel{v_{l}}{\bullet} \ar[uurr]_{a_{l}} & & & &\stackrel{w_{1}}{\bullet} \\ 
 }$   
%%
%% \item    $\xymatrix@-1pc{ 
%% \stackrel{v_1}{\bullet} & & \ar[ddl]_{a_{l'}}  \ar[ll]_{a_1}  \stackrel{m}{\bullet} \ar@/^/[ddr]^{c_l} \ar@/_/[ddr]_{s_l} \ar@/^/[rr]^{c_l} \ar@/_/[rr]_{s_{1}}& & \stackrel{w_1}{\bullet}\\ 
%% \vdots  &  & & & .^{.^.}   \\
%% & \stackrel{v_{l'}}{\bullet} &   & \stackrel{w_{l}}{\bullet} & \\  
%% }$ 
%% 
%% Case 11 [below]
% \item\label{item:twopathways-case11}   $\xymatrix@-1pc{ 
% \stackrel{v_1}{\bullet} \ar[rr]^{a_m}   & &   \stackrel{m}{\bullet} \ar@/^/[ddr]^{c_l} \ar@/_/[ddr]_{s_l} %\ar@/^/[rr]^{c_l} \ar@/_/[rr]_{s_{1}}& & \stackrel{w_1}{\bullet}\\ 
% \vdots  &  & & & .^{.^.}   \\
% & \stackrel{v_{l'}}{\bullet}\ar[uur]^{a_{l'}}  &   & \stackrel{w_{l}}{\bullet} & \\  
% }$ 
%% Case 12  [below] 
% \item\label{item:twopathways-case12}  $\xymatrix@-1pc{ 
% \stackrel{v_1}{\bullet} \ar[rr]^{a_m}  & & \stackrel{m}{\bullet} & & \stackrel{w_1}{\bullet} \ar@/^/[ll]^{s_1} %\ar@/_/[ll]_{c_{1}} \\ 
% \vdots & & & & .^{.^.}\\
% & \stackrel{v_{l'}}{\bullet} \ar[uur]^{a_{l'}}  &  & \stackrel{w_{l}}{\bullet} \ar@/^/[luu]^{s_l} \ar@/_/[luu]_{c_l}& %\\  
% }$  
% % Case 13 [below] 
% \item\label{item:twopathways-case13}  $\xymatrix@-1pc{ 
% \stackrel{v_1}{\bullet}   & &  \ar[ll]_{a_1} \ar[ddl]_{a_{l'}}  \stackrel{m}{\bullet} & & \stackrel{w_1}{\bullet} %\ar@/^/[ll]^{s_1} \ar@/_/[ll]_{a_{m}} \\ 
% \vdots & & & & .^{.^.}\\
% & \stackrel{v_{l'}}{\bullet}  &  & \stackrel{w_{l}}{\bullet} \ar@/^/[luu]^{s_l} \ar@/_/[luu]_{c_l}& \\  
% }$   
 % Case 14 [below] 
 \item\label{item:twopathways-case14}    $\xymatrix@-1pc{ 
 \stackrel{v_1}{\bullet} \ar[ddrr]^{a_1} & &  & & \stackrel{w_1}{\bullet} \\ 
  \vdots    & &  & &  \vdots \\  
 \stackrel{v_{l}}{\bullet}\ar[rr]^{a_{l}} & & \stackrel{m}{\bullet} \ar[ddl]_{c_k} \ar[ddr]^{c_{k'+2}} \ar@/^/[rruu]^{c_1} \ar@/_/[rruu]_{c_2}  \ar@/^/[rr]^{c_{k'}} \ar@/_/[rr]_{c_{k'+1}}& & \stackrel{w_{k''}}{\bullet} \\ 
 & &  & &  \\  
  & \stackrel{w_{k'''}}{\bullet} &\cdots  & \stackrel{w_{k''+1}}{\bullet} &   \\ 
 }$       
 % Case 15 [below] 
 \item\label{item:twopathways-case15}    $\xymatrix@-1pc{ 
 \stackrel{v_{l'}}{\bullet} \ar[ddrr]^{a_{l'}} & &  & & \stackrel{v_{l'+1}}{\bullet} \ar@/^/[lldd]^{a_{l'+2}} \ar@/_/[lldd]_{a_{l'+1}}  \\ 
  \vdots    & &  & &  \vdots \\  
 \stackrel{v_{1}}{\bullet}\ar[rr]^{a_{1}} & & \stackrel{m}{\bullet} \ar[ddl]_{c_1} \ar[ddr]^{c_{k}} & & \stackrel{v_{l''}}{\bullet} \ar@/^/[ll]^{a_l} \ar@/_/[ll]_{a_{l-1}} \\ 
 & &  & &  \\ 
  & \stackrel{w_1}{\bullet} &\cdots  & \stackrel{w_{k}}{\bullet} &   \\ 
 }$ 
 \end{enumerate}     
\end{multicols}    

The following argument holds for all four cases. 
\begin{comment}
Next, consider Case 
%\eqref{item:twopathways-case8}, 
%\eqref{item:twopathways-case9}, 
%and  
\eqref{item:twopathways-case10}.  
\end{comment}   
Relabel the arrows in the following way:  
write  
$a_1,\ldots, a_l$ if $ha_i=m$,   
$c_1,\ldots, c_k$ if $tc_j=m$, and   
$\zeta$ if $\zeta$ is the loop at $m$.    
Let $\{ q_{\phi}\}_{0\leq \phi\leq Q_1-l-k-1}$  
(or $\{ q_{\phi}\}_{0\leq \phi\leq Q_1-l-k}$ if there is no loop at vertex $m$) be all the other arrows of $Q$,  
where $q_{0} :=\varnothing$.  
Write general representations of $a_{\alpha}$ as $A_{\alpha} =({}_{(\alpha)}a_{st})$,
$c_{\gamma}$ as $C_{\gamma}=({}_{(\gamma)}c_{st})$, and 
$\zeta$ as $\Xi = (\zeta_{st})$, and let $({}_{(\phi)}q_{st})$ be a general representation of $q_{\phi}$. 
Write 
\[ 
f = \sum_{\rho+|K'|+|\Gamma| \leq d} 
\zeta_{ij}^{\rho} 
\prod_{\alpha=1}^{l} {}_{(\alpha)}a_{ij}^{k_{\alpha}} 
\prod_{\gamma=1}^{k} {}_{(\gamma)}c_{ij}^{\mu_{\gamma}} \hspace{1mm}
f_{ij,\rho, K', \Gamma}, 
\] 
where $K'=(k_1,\ldots, k_l)$, 
$\Gamma = (\mu_1,\ldots, \mu_k)$, 
$|K'|=\displaystyle{\sum_{\alpha=1}^{l} k_{\alpha}}$, 
$|\Gamma| = \displaystyle{ \sum_{\gamma=1}^{k} \mu_{\gamma}}$, 
and 
$f_{ij,K',\Gamma} \in  
\mathbb{C}[\{ {}_{(\alpha)}a_{st}$, ${}_{(\gamma)}c_{st}$, $\zeta_{st}$, ${}_{(\phi)}q_{st}$ : 
$(s,t)\not=(i,j)$ 
$\mbox{ and }$ 
$\alpha\not\in \{ 1,\ldots, l\}$
$\mbox{ and }$ 
$\gamma\not\in \{1,\ldots, k \}
\}] =: R_1$. 
Then 
\[ 
\begin{aligned} 
0 &= u_{ij}^{-1}.f-f 
= \sum_{\rho+ |K'|+|\Gamma| \leq d}  
(\zeta_{ij}+(\zeta_{jj}-\zeta_{ii})u)^{\rho} 
\prod_{\alpha=1}^{l} ({}_{(\alpha)}a_{ij}+{}_{(\alpha)}a_{jj}u )^{k_{\alpha}} 
\prod_{\gamma=1}^{k} ({}_{(\gamma)}c_{ij}-{}_{(\gamma)}c_{ii}u )^{\mu_{\gamma}}
f_{ij,\rho, K', \Gamma} \\
   &\hspace{4mm}-
\sum_{\rho+ |K'|+|\Gamma| \leq d} 
\zeta_{ij}^{\rho}
\prod_{\alpha=1}^{l} {}_{(\alpha)}a_{ij}^{k_{\alpha}} 
\prod_{\gamma=1}^{k} {}_{(\gamma)}c_{ij}^{\mu_{\gamma}}\hspace{1mm}
f_{ij,\rho, K', \Gamma}  \\  
&= \sum_{1\leq \rho+ |K'|+|\Gamma| \leq d}  \sum_{ \stackrel{\tau\leq \rho}{
r_{\alpha}\leq k_{\alpha}, s_{\gamma}\leq \mu_{\gamma}
}}  
\binom{\rho}{\tau} 
  \binom{k_1}{r_1} \binom{k_2}{r_2} \cdots \binom{k_{l}}{r_{l}} 
  \binom{\mu_1}{s_1} \binom{\mu_2}{s_2} \cdots \binom{\mu_{k}}{s_{k}} \cdot \\  
  &\hspace{4mm}\cdot \zeta_{ij}^{\rho-\tau} 
  \prod_{\alpha=1}^{l} {}_{(\alpha)}a_{ij}^{k_{\alpha}-r_{\alpha}}
  \prod_{\gamma=1}^{k} {}_{(\gamma)}c_{ij}^{\mu_{\gamma}-s_{\gamma}} 
  \cdot u^{\tau+ |R|+|S|} 
  (\zeta_{jj}-\zeta_{ii})^{\tau}  
  \prod_{\alpha=1}^{l} ({}_{(\alpha)}a_{jj})^{r_{\alpha}} 
  \prod_{\gamma=1}^{k} (-{}_{(\gamma)}c_{ii})^{s_{\gamma}}
  f_{ij,\rho, K',\Gamma}, \\ 
\end{aligned}  
\] 
where $|R|=\displaystyle{\sum_{\alpha=1}^{l}r_{\alpha}}$ and 
$|S|=\displaystyle{\sum_{\gamma=1}^{k}s_{\gamma}}$,  
and we see that 
$\{\zeta_{ij}^{\rho-\tau} 
\displaystyle{ \prod_{\alpha=1}^{l}{}_{(\alpha)}a_{ij}^{k_{\alpha}-r_{\alpha}}
\prod_{\gamma=1}^{k} {}_{(\gamma)}c_{ij}^{\mu_{\gamma}-s_{\gamma}}
u^{\tau+ |R|+|S|} }:    
\tau\leq \rho,  
r_{\alpha}\leq k_{\alpha}, s_{\gamma}\leq \mu_{\gamma} \mbox{ for all }1\leq \alpha\leq l  
\mbox{ and for all }1\leq \gamma\leq k \}$   
is linearly independent over $R_1$.   
This implies that $f_{ij,\rho, K',\Gamma}$ $=$ $0$ 
%% 
%% below: is it AND or OR?  
whenever $|K'|\geq 1$, $|\Gamma|\geq 1$, or $\rho\geq 1$,   
which contradict our choices of $(i,j)$ and $m$. We conclude that 
 $f\in \mathbb{C}[\mathfrak{t}^{\oplus Q_1}]$.   
 \end{proof}

\section{Proof of Theorem~\ref{theorem:two-paths-max-quiver-semi-invariants-framed}}

Let $\lambda=(\lambda_1\geq \lambda_2\geq \ldots \geq \lambda_l\geq 0)$ be a {\em partition} of size 
$|\lambda|$ $=$ $\sum_{i=1}^{l} \lambda_i$.  
One identifies to $\lambda$ a left-justified shape of $l$ rows of boxes of length $\lambda_1$, $\lambda_2$, $\ldots$, $\lambda_l$,  
which is called the {\em Young diagram} associated to $\lambda$. 
A {\em (Young) filling} or a {\em Young tableau} 
of $\lambda$ assigns a positive integer to each box of Young diagram. 

\begin{example}\label{example:young-diagram}  
Associated to $\lambda=(5,3,3,1)$ is the Young diagram  
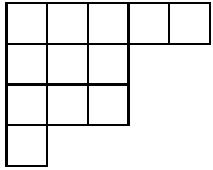
\[ 
\ydiagram{5,3,3,1} 
\] 
and a Young tableau    
\ytableausetup{mathmode, boxsize=2em} 
 \begin{ytableau} 
 4& 1& 2 & 3&3 \\ 
 3& 4& 5\\
 2& 5& 6\\    
 1 
 \end{ytableau}.    
\end{example}

\begin{definition}\label{definition:standard-form-bitableau}  
A Young tableau is   
{\em normal}   
if the entries in each row are strictly increasing from left to right. 
It is called {\em standard} if it is normal and the entries in each column are nondecreasing from top to bottom. 
A {\em bitableau} $J|I$ is a pair of Young tableaux $J$ and $I$ having the same shape and the bitableau is called {\em standard} if both $J$ and $I$ are standard Young tableaux. 
\end{definition}  

\begin{remark}\label{remark:fulton-and-harris-standard-semistandard} 
In Fulton and Harris (\cite{MR1153249}), a Young tableau is called {\em standard} if the entries in each row and column are strictly increasing from left to right and from top to bottom, and a Young tableau is called {\em semistandard} if the entries in each row are nondecreasing from left to right while the entries in each column are strictly increasing from top to bottom.  We will not use their definition in this paper. 
\end{remark}

The {\em bideterminant} $(J|I)$  
of a bitableau is defined in the following way: 
%calculate the determinant of the minors of a matrix  
the positive integer entries in $J$ in a fixed row correspond to the rows of a matrix while the positive integer entries in   
$I$ in the same row correspond to the columns of a matrix. Take the determinant of these minors of a matrix and repeat for each row in $J|I$ to obtain the bideterminant 
$(J|I)$.   

Note that the bideterminant $(J|I)$ associated to a bitableau $J|I$ is a product of minors of a matrix, where $J$ are the row indices and $I$ are the column indices. 
Furthermore, a matrix (or a product of matrices) must be specified when calculating the bideterminant associated to a bitableau; 
such specification will be denoted on the lower-right corner of each row of the bitableau (and the bideterminant), i.e., see 
\eqref{eq:bitableau-general-form-for-framed-quivers} 
and  
\eqref{eq:bideterminant-general-form-for-framed-quivers}.  

\begin{example} 
Consider the two Young tableaux:  
\[ J = 
\ytableausetup{mathmode, boxsize=2em} 
 \begin{ytableau} 
 1& 2& 4 \\ 
 2& 4    \\
 1& 4    \\    
 \end{ytableau}
%\begin{array}{ccc} 
%1 & 2 & 4 \\ 
%2 & 4 &   \\ 
%1 & 4 &   \\
%\end{array} 
\hspace{2mm} 
\mbox{ and } 
\hspace{2mm}
I = 
\ytableausetup{mathmode, boxsize=2em} 
 \begin{ytableau} 
 1& 2 & 4 \\ 
 2& 4 \\
 2& 5 \\    
 \end{ytableau}. 
%\begin{array}{ccc} 
%1 & 2 & 4 \\ 
%2 & 4 &   \\ 
%2 & 5 &   \\  
%\end{array}. 
\] 
$J$ is normal but not standard (since the entries in the first column are not nondecreasing when reading from top to bottom), while $I$ is standard. 
\end{example}

From this point forward, we will not draw a box around each entry of a Young (bi)tableau or the bideterminant of a bitableau.

Let $Q^{\dagger}$ be a quiver with one framed vertex labelled as $1'$  
and all other nonframed vertices labelled as $1$, $2$, $\ldots$, $|Q_0|$,  
and the arrows labelled as $a_0$, $a_1$, $\ldots$, $a_{|Q_1|}$,  
where $ta_0=1'$ and $ha_0=1$.  
Let $A_{\phi_u}$ be a general representation of the arrow $a_{\phi_u}$. 
The product 
$A_{\phi_1}A_{\phi_{2}} \cdots A_0$  
of general matrices is associated to the quiver path 
$a_{\phi_1}a_{\phi_{2}} \cdots a_0$, which begins at the framed vertex.   
Moreover,   
$A_{\phi_1}A_{\phi_{2}} \cdots A_0$   
is uniquely associated to the sequence 
\begin{equation}\label{equation:unique-sequence-ass-to-prod-of-matrices}
\mathbf{\Phi} := [\phi_1, \phi_2, \ldots, 0] 
\end{equation}  
of integers obtained by reading the indices of 
$A_{\phi_1}A_{\phi_{2}} \cdots A_0$.   
Consider all 
$\mathbf{\Phi}$ whose quiver paths begin at the framed vertex;  
we will fix a partial ordering $\leq$ on these finite sequences of nonnegative integers. We say 
\[ 
\mathbf{\Phi} := [\phi_1, \phi_2, \ldots, 0] \leq 
[\psi_1, \psi_2, \ldots, 0] =: \mathbf{\Psi} 
\]  
if the number of components in 
$\mathbf{\Phi}$ 
is less than the number of components in 
$\mathbf{\Psi}$, 
or 
if the number of components in  
$\mathbf{\Phi}$ 
equals the number of components in 
$\mathbf{\Psi}$ 
and the right-most nonzero entry in 
$\mathbf{\Psi}-\mathbf{\Phi}$ is positive when subtracting component-wise.

\begin{definition}\label{definition:bideterminant-product-of-general-matrices-framed-quiver}  
The $s^{th}$ row of the bitableau  
\begin{equation}\label{eq:bitableau-general-form-for-framed-quivers}
 \begin{matrix}
   j_1^{(\phi)} & j_2^{(\phi)} & \ldots & j_{v_1}^{(\phi)}  &| 
  i_1^{(\phi)}& i_2^{(\phi)} &\ldots & {i_{v_1}^{(\phi)}}_{A_{\phi_1}\cdots  A_0} \\
   j_1^{(\psi)} & j_2^{(\psi)} & \ldots & j_{v_2}^{(\psi)}  &| 
  i_1^{(\psi)}& i_2^{(\psi)} &\ldots & { i_{v_2}^{(\psi)}}_{A_{\psi_1}\cdots  A_0} \\
       & \vdots&        &    &           &  \vdots   &       &                   \\
   j_1^{(\mu)} & j_2^{(\mu)} & \ldots & j_{v_l}^{(\mu)}  &| 
  i_1^{(\mu)}& i_2^{(\mu)} &\ldots & {i_{v_l}^{(\mu)}}_{A_{\mu_1}\cdots  A_0} \\
  \end{matrix} 
\end{equation} 
is defined to be the bitableau associated to rows 
$j_1^{(\psi)}$, $j_2^{(\psi)}$, $\ldots$, $j_{v_s}^{(\psi)}$ and columns 
$i_1^{(\psi)}$, $i_2^{(\psi)}$, $\ldots$, $i_{v_s}^{(\psi)}$ in the product $A_{\psi_1}A_{\psi_2}\cdots A_0$ of general matrices. 
The bideterminant 
\begin{equation}\label{eq:bideterminant-general-form-for-framed-quivers} 
\begin{matrix}
   (j_1^{(\phi)} & j_2^{(\phi)} & \ldots & j_{v_1}^{(\phi)}  &| 
  i_1^{(\phi)}& i_2^{(\phi)} &\ldots & i_{v_1}^{(\phi)})_{A_{\phi_1}\cdots  A_0} \\
   (j_1^{(\psi)} & j_2^{(\psi)} & \ldots & j_{v_2}^{(\psi)}  &| 
  i_1^{(\psi)}& i_2^{(\psi)} &\ldots & i_{v_2}^{(\psi)})_{A_{\psi_1}\cdots  A_0} \\
       & \vdots&        &    &           &  \vdots   &       &                   \\
   (j_1^{(\mu)} & j_2^{(\mu)} & \ldots & j_{v_l}^{(\mu)}  &| 
  i_1^{(\mu)}& i_2^{(\mu)} &\ldots & i_{v_l}^{(\mu)})_{A_{\mu_1}\cdots  A_0}  
\end{matrix}
\end{equation}
is the product of bideterminants of the form 
\[ 
\begin{matrix}
(j_1^{(\psi)} & j_2^{(\psi)} & \ldots & j_{v_s}^{(\psi)}  &| 
  i_1^{(\psi)}& i_2^{(\psi)} &\ldots & i_{v_s}^{(\psi)})_{A_{\psi_1}A_{\psi_2}\cdots  A_0}  
  \end{matrix} 
\]  
obtained by taking the determinant of minors of rows  
$j_1^{(\psi)}, j_2^{(\psi)},\ldots, j_{v_s}^{(\psi)}$ and columns  
$i_1^{(\psi)}, i_2^{(\psi)},\ldots, i_{v_s}^{(\psi)}$ in the product 
$A_{\psi_1} A_{\psi_2} \cdots A_0$ of general matrices.   
\end{definition}   
Note that the $s^{th}$ row of \eqref{eq:bitableau-general-form-for-framed-quivers} is associated to the sequence $\mathbf{\Psi}$ of integers.  
%(cf. \eqref{equation:unique-sequence-ass-to-prod-of-matrices}).  
We say the bitableau \eqref{eq:bitableau-general-form-for-framed-quivers}  
is in {\em block standard form}  
if the sequence $\mathbf{\Phi}$ of integers associated to each row of the bitableau is in nondecreasing order (with respect to the partial ordering $\leq$ defined earlier in this section) when ascending down the rows, and in the case the sequence $\mathbf{\Phi}$ 
for multiple rows are identical, then these rows are in standard form as defined in Definition~\ref{definition:standard-form-bitableau}.  
In the case that the bitableau \eqref{eq:bitableau-general-form-for-framed-quivers} is in block standard form, we will interchangeably say the bideterminant (associated to the bitableau) is in block standard form.

\begin{example}\label{example:standard-form-bitableau} 
Let 
\[  
A_0 = \begin{pmatrix}   
 x_{11}& x_{12} \\  
 x_{21} &x_{22}   
\end{pmatrix}
 \mbox{ and } 
 A_1= \begin{pmatrix} 
  a_{11}&a_{12} \\   
   0    &a_{22}  
 \end{pmatrix}. 
\] 
Then 
\[ 
\begin{matrix} 
 ( 2& 		  |& 1 & 		&)_{A_0} \\  
 (2&  		  |& 2 & 		&)_{A_0} \\
 (1& 		2 |& 1 & 2  &)_{A_1A_0} \\ 
 ( 2& 		  |& 1 & 		&)_{A_1A_0}  
\end{matrix}
\] 
is a bideterminant in block standard form, 
and the polynomial associated to the bideterminant is  
\[   
\begin{aligned}  
x_{21} x_{22} \det(A_1A_0)\cdot (A_1A_0)_{2,1}&= 
x_{21}x_{22} a_{11}a_{22}(x_{11}x_{22}-x_{12}x_{21})a_{22} x_{21} \\ 
&= x_{21}^2 x_{22}  a_{11}a_{22}^2 (x_{11}x_{22} - x_{12} x_{21}).\\ 
\end{aligned}
\] 
\end{example}

The following is Theorem 13.1 in \cite{MR1489234}. 

\begin{theorem}\label{theorem:standard-basis-theorem-Grosshans}
Let $R= \mathbb{C}[\{ x_{ij}: 1\leq i\leq n, 1\leq j\leq m\}]$. 
Then bideterminants of standard bitableaux form a basis over $\mathbb{C}$ of $R$. 
\end{theorem}

%{\em Straightening} 
%a bideterminant means expressing the bitableau associated to the bideterminant in terms of the standard %basis. 

\begin{lemma}\label{lemma:borel-invariants-on-bideterminants}  
Let invertible upper triangular matrices $B\subseteq GL_n(\mathbb{C})$ act on $M_{n\times m}$ via left translation. 
Consider the character  
$\chi_p(b)$ $=$ $\displaystyle{\prod_{i=p}^n b_{ii}}$ of $B$ and let 
$f=(p\:\: p+1\:\cdots n | i_1 \cdots i_{n-p+1})\in \mathbb{C}[M_{n\times m}]$ 
where $1\leq i_1 < i_2 < \ldots < i_{n-p+1}\leq m$. 
Then $b.f=\chi_p(b)f$. 
\end{lemma}

\begin{proof} 
 Write $b=tu$, where $t=(t_{ii})\in T$ and $u\in U$, $T$ is the maximal torus in $B$ and $U$ is the maximal unipotent subgroup of $B$. 
Then 
$t.f = \displaystyle{\prod_{i=p}^n t_{ii} f} = \chi_p(t)f$ 
and for the subgroup $U_{i,j}$ which has $1$ along the main diagonal, the variable $u$ in the $(i,j)$-entry and $0$ elsewhere,
we will show that $U_{i,j}$ fixes $f$. 
So let $\widehat{u}\in U_{i,j}$. 
Then since $\widehat{u}.f(x) = f(\widehat{u}^{-1}.x) = f(\widehat{u}^{-1}x)$, 
\[
 \widehat{u}.x_{st} = 
\begin{cases} 
x_{it}-x_{jt}u &\mbox{ if } s =i, \\ 
x_{st} &\mbox{ otherwise}. \\ 
\end{cases}
\]
So $\widehat{u}.f=f-u(p\:\: p+1\cdots j \cdots j \cdots n| i_1 \cdots i_{n-p+1})=f$ if $p\leq i$
and $\widehat{u}.f=f$ if $p>i$. 
This concludes the proof. 
\end{proof}

The following lemma generalizes Lemma~\ref{lemma:borel-invariants-on-bideterminants}. 

\begin{lemma}\label{lemma:borel-invariants-on-equioriented-acylic-quiver}
Consider the {\em equioriented} (all arrows are pointing in the same direction) quiver 
\[ 
\xymatrix@-1pc{
\stackrel{1'}{\circ} \ar[rr]^{a_0} & & \stackrel{1}{\bullet} \ar[rr]^{a_1} & & \stackrel{2}{\bullet} 
&\ldots & \stackrel{r}{\bullet} \ar[rr]^{a_r} & & \stackrel{r+1}{\bullet},   
}
\]  
where $\beta_{1'}=m$ and $\beta_i = n$ for $1\leq i\leq r+1$. 
%the dimension of the vector space at the framed vertex is $m$ and the dimension of the vector space at nonframed vertices is $n$. 
% Impose the complete standard filtration $F^{\bullet}$ of vector spaces at the nonframed vertices. 
Let $\mathbb{U}_{\beta}\subseteq B^{r+1}$ be the product of largest unipotent subgroups and let  
$A_m$ be a general matrix associated to arrow $a_m$ for each $0\leq m\leq r$.  
Then $F^{\bullet}Rep(Q,\beta)=M_{n\times m}\oplus \mathfrak{b}^{\oplus r}$ 
and standard bideterminants of the form 
\begin{equation}\label{eq:standard-bideterminant-one-row-only-lemma}
(p\:\: p+1 \: \cdots \: n \: | \: i_1 \: i_2 \: \cdots \: i_{n-p+1} )_{A_{m}\cdots A_0}, 
\hspace{4mm} \mbox{ where }  1\leq p\leq n \mbox{ and } 0 \leq m\leq r, 
\end{equation}  
are $\mathbb{U}_{\beta}$-invariant polynomials. 
\end{lemma}  

%Note that the polynomial \eqref{eq:standard-bideterminant-one-row-only-lemma} is a row of \eqref{eq:bideterminant-general-form-for-framed-quivers}.   
  
\begin{proof}   
For $u=(u_1,\ldots, u_{r+1})\in \mathbb{U}_{\beta}$ and  
$(A_0,A_1,\ldots, A_r)\in M_{n\times m}\oplus \mathfrak{b}^{\oplus r}$, 
\[ 
u.(A_0,A_1,\ldots, A_r) = 
(u_1 A_0, u_2 A_1 u_1^{-1},\ldots, u_{\alpha+1}A_{\alpha}u_{\alpha}^{-1},\ldots, u_{r+1}A_r u_r^{-1}). 
\] 
We will write the entries of the product $A_{\alpha}\cdots A_0$ 
of matrices as $({}_{(\alpha)}y_{st})$, where $({}_{(0)}y_{st})=(x_{st})\in M_{n\times m}$. 
Then for the subgroup $U_{ij}$ of $\mathbb{U}_{\beta}$ which is    
\[ 
\begin{aligned} 
U_{ij} =   \{ u_{ij} = (\I,\ldots, \widehat{u}_m,\ldots, \I):   
\widehat{u}_m &\mbox{ is the matrix with the variable $u$ in the }(i,j)\mbox{-}entry, \mbox{ where }i<j, \\ 
&1 \mbox{ along the diagonal entries, and }
0 \mbox{ elsewhere}     \},    \\ 
\end{aligned}
\] 
$u_{ij}\in U_{ij}$ acts on $({}_{(\alpha)} y_{st}) \in A_{\alpha}\cdots A_1 A_0$ as follows: $U_{ij}$ changes the coordinate polynomial ${}_{(\alpha)}y_{st}$ via  
\[   
u_{ij}.{}_{(\alpha)}y_{st} = 
\begin{cases} 
{}_{(m-1)}y_{it}-{}_{(m-1)}y_{jt}u &\mbox{ if }\alpha = m-1 \mbox{ and }s=i, \\ 
{}_{(\alpha)}y_{st} &\mbox{ otherwise. } \\ 
\end{cases} 
\] 
So for $f=(p\:\: p+1\:\cdots \: n| i_1\: i_2 \: \cdots i_{n-p+1})_{A_{\alpha}\cdots A_0}$, 
\[ 
u_{ij}.f=
\begin{cases} 
f-u(p\:\: p+1\:\cdots j\:\cdots \:j \cdots \: n| i_1\: i_2 \: \cdots i_{n-p+1})_{A_{m-1}\cdots A_0} = f  &\mbox{ if } \alpha=m-1 \mbox{ and } p\leq i, \\ 
f &\mbox{ otherwise. }    \\ 
\end{cases}
\] 
Thus standard bideterminants of the form 
$(p\:\: p+1 \: \cdots \: n \: | \: i_1 \: i_2 \: \cdots \: i_{n-p+1} )_{A_{m}\cdots A_0}$
are $\mathbb{U}_{\beta}$-invariant polynomials, 
where $1\leq p\leq n$ 
 and 
$0 \leq m\leq r$. 
\end{proof}

\begin{definition}\label{definition:bideterminant-poly-for-theorem}
Assume 
$\beta=(n,\ldots, n,m)\in \mathbb{Z}_{\geq 0}^{Q_0^{\dagger}}$,  
where $m$ is associated to the framed vertex. 
Define 
\begin{equation}\label{eq:bideterminant-block-standard-form-generic-notation}
(J|I)_{\mathbf{\Phi } \mathbf{\Psi} \cdots \mathbf{\Gamma}} := 
  \begin{matrix}
   (j_1 & j_1+1 & \ldots & n  &| i_1^{(\phi)}& i_2^{(\phi)} &\ldots & i_{n-j_1+1}^{(\phi)})_{A_{\phi_1} \cdots  A_0} \\
   (j_2 & j_2+1 & \ldots & n  &| i_1^{(\psi)}& i_2^{(\psi)} &\ldots & i_{n-j_2+1}^{(\psi)})_{A_{\psi_1} \cdots  A_0} \\ 
       & \vdots&        &    &           &  \vdots   &       &                   \\
   (j_l & j_l+1 & \ldots & n  &| i_1^{(\mu)}& i_2^{(\mu)} &\ldots & i_{n-j_l+1}^{(\mu)})_{A_{\mu_1} \cdots  A_0}  
  \end{matrix}  
\end{equation}  
as the bideterminant in block standard form, 
where  
$A_{\psi_1}\cdots A_0$ 
is a general representation of the quiver path $a_{\psi_1}\cdots a_0$ 
which begins at the framed vertex.  
\end{definition}

Note that each sequence $\mathbf{\Phi_{}}$ of integers associated to each row of   \eqref{eq:bideterminant-block-standard-form-generic-notation}  
corresponds to a general representation of a quiver path that begins at the framed vertex (this is important as this will imply the uniqueness of \eqref{eq:bideterminant-block-standard-form-generic-notation}:   
 if 
 $(J|I)_{\mathbf{\Phi_{}}\mathbf{\Psi_{}} \cdots \mathbf{\Gamma_{}}} 
 =  
 (J'|I')_{\mathbf{\Phi_{}'}\mathbf{\Psi_{}'} \cdots \mathbf{\Gamma_{}'}}$ 
 where $(J|I)_{\mathbf{\Phi_{}}\mathbf{\Psi_{}} \cdots \mathbf{\Gamma_{}}}$ and  
 $(J'|I')_{\mathbf{\Phi_{}'}\mathbf{\Psi_{}'} \cdots \mathbf{\Gamma_{}'}}$  
 are in block standard form,  
 then $J=J'$, $I=I'$, and  
 $\mathbf{\Phi}=\mathbf{\Phi'}$, 
 $\mathbf{\Psi_{}}=\mathbf{\Psi_{}'}$,
 $\ldots$, 
 $\mathbf{\Gamma} = \mathbf{\Gamma'}$). 
 The following proof is a generalization of the proof of Theorem 13.3 in \cite{MR1489234}.

\begin{proof}  
By Theorem~\ref{theorem:two-paths-max-quiver-semi-invariants}, it suffices to find all invariants for paths starting at a framed vertex. 
Let ${}_{(\alpha)}a_{st}$ be the entries of a general matrix $A_{\alpha}$ and let $x_{st}$ be the entries of a general matrix $A_0$. 
Suppose $f$ is a $\mathbb{U}_{\beta}$-invariant polynomial.  
Without loss of generality, if 
$f(x_{st},{}_{(\alpha)}a_{st})=g(x_{st},{}_{(\alpha)}a_{st}) +h({}_{(\alpha)}a_{ii})$, 
where all the monomials of $g$ are divisible by some $x_{st}$ for some $s$ and $t$ and $h\in \mathbb{C}[\mathfrak{t}^{\oplus Q_1}]$, 
then we only consider $g$ by subtracting off $h$ since we have already proved that $h=h({}_{(\alpha)}a_{ii})$ is an invariant polynomial.  
By applying Lemma~\ref{lemma:borel-invariants-on-equioriented-acylic-quiver} to each row of  
\eqref{eq:bideterminant-block-standard-form-generic-notation}, we see that \eqref{eq:affine-dk-quiver-general-n-inv-poly}  
is a $\mathbb{U}_{\beta}$-invariant polynomial.  
Now suppose there exists a polynomial not in $\mathbb{C}[\mathfrak{t}^{\oplus Q_1}]$ or not of the form \eqref{eq:affine-dk-quiver-general-n-inv-poly}  
 which is in $\mathbb{C}[F^{\bullet}Rep(Q,\beta)]^{\mathbb{U}_{\beta}}$.
 That is, 
 suppose there exists $f\in \mathbb{C}[F^{\bullet}Rep(Q,\beta)]$ fixed by $\mathbb{U}_{\beta}$ 
 with its monomials divisible by $x_{st}$ for some $s$ and $t$     
 which cannot be written as   
 \eqref{eq:affine-dk-quiver-general-n-inv-poly}. 
Let $j\leq n-1$ be the biggest integer which satisfies the following: 
\begin{center}\label{center:unique-choice-of-j-affine-Ar-dynkin}
there exists a $\mathbb{U}_{\beta}$-invariant 
$F\in \mathbb{C}[F^{\bullet}Rep(Q,\beta)]$ such that when $F$ is written in terms of the block standard basis, i.e., 
$F=\sum_{\nu}g_{\nu}({}_{(\alpha)}a_{st})(J_{\nu}|I_{\nu})_{
\mathbf{\Phi_{}}\mathbf{\Psi_{}} \cdots \mathbf{\Gamma_{}}}$
with each $(J_{\nu}|I_{\nu})_{\mathbf{\Phi_{}}\mathbf{\Psi_{}} 
\cdots \mathbf{\Gamma_{}}}$ 
a standard Young bideterminant in block standard form and each $g_{\nu}\not=0$, 
then there exists a $v$ and a row in $J_v$ where $j$ is not followed by $j+1$. 
Let us label this choice of $j$ as 
$(\dagger)$. 
\end{center}

Let $a_{\phi}$ be the arrow associated to a general representation $A_{\phi}$. 
Writing $ha_{\phi}$ to be the head of the arrow $a_{\phi}$, let 
$U_{j,j+1}$ be the subgroup consisting of matrices of the form 
$u_{j,j+1} = (\I,\ldots, u_{ha_{\phi}},\ldots, \I)$ where $ha_{\phi}$ has diagonal entries $1$, the variable $-u$ in $(j,j+1)$-entry, and $0$ elsewhere.  
Let's write the entries of the product $A_{\alpha}\cdots A_0$ of matrices as $y_{st}$. 
Then for $u_{j,j+1}\in U_{j,j+1}$, 
\[ 
u_{j,j+1}.y_{st} = 
\begin{cases} 
y_{jt}+u y_{j+1,t} &\mbox{ if }\alpha=\phi \mbox{ and } s=j, \\ 
y_{st} 						&\mbox{ otherwise} \\ 
\end{cases} 
\] 
since $A_{\alpha}\cdots A_0$ is a general representation of the quiver path $a_{\alpha}\cdots a_0$. 
To explain further, if the path $a_{\alpha}\cdots a_0$ includes $a_{\phi}$ 
somewhere strictly in the middle of the path, i.e., 
$a_{\alpha}\cdots a_0$ $=$ $a_{\alpha}\cdots a_{\phi}\cdots a_0$, then although $u_{ha_{\phi}}$ 
acts by left multiplication on $A_{\phi}$, 
$u_{ha_{\phi}}$ acts by right (inverse) multiplication on the general representation of the arrow in the path immediately following $a_{\phi}$ (this is the arrow which is immediately to the left of $a_{\phi}$ 
in the concatenation of the arrows $a_{\alpha}\cdots a_0$). 
Thus, the action by $u_{ha_{\phi}}$ is canceled.  
So for any $\alpha$,  
$u$ fixes every minor of the form   
\[  
(\cdots j\:\: j+1 \cdots | \cdots )_{\mathbf{\Phi_{ }}\mathbf{\Psi_{ }} \cdots \mathbf{\Gamma_{}}}. 
\]

Now let us write 
\[ 
F = \sum_{\nu} g_{\nu}({}_{(\alpha)} a_{st})(J_{\nu}|I_{\nu})_{
\mathbf{\Phi_{ }}\mathbf{\Psi_{}} \cdots \mathbf{\Gamma_{ }}}
+ \sum_{\gamma} g_{\gamma}({}_{(\alpha)} a_{st})(J_{\gamma}|I_{\gamma})_{
\mathbf{\Phi_{ }}\mathbf{\Psi_{ }} \cdots \mathbf{\Gamma_{ }}}, 
\]   
with the following properties:  
\begin{itemize}  
\item the $g_{\nu}$ are nonzero, 
\item there exists at least one row in each $J_{\nu}$ which contains $j$ but not $j+1$,
\item if $j$ appears in any row of $J_{\gamma}$, then so does $j+1$, 
\item the $(J_{\nu}|I_{\nu})_{
\mathbf{\Phi_{ }}\mathbf{\Psi_{ }} \cdots \mathbf{\Gamma_{ }}}$ and 
$(J_{\gamma}|I_{\gamma})_{
\mathbf{\Phi_{ }}\mathbf{\Psi_{ }} \cdots \mathbf{\Gamma_{ }}}$ are unique. 
\end{itemize} 
By Lemma~\ref{lemma:borel-invariants-on-equioriented-acylic-quiver}, 
$\mathbb{U}_{\beta}$ fixes each row of $(J_{\gamma}|I_{\gamma})_{
\mathbf{\Phi_{ }}\mathbf{\Psi_{ }} \cdots \mathbf{\Gamma_{ }}}$.   
Since  
$(J_{\gamma}|I_{\gamma})_{\mathbf{\Phi_{ }}\mathbf{\Psi_{ }} \cdots 
\mathbf{\Gamma_{ }}}$  
is the product of the rows in the bideterminant, 
$\mathbb{U}_{\beta}$ 
fixes 
$(J_{\gamma}|I_{\gamma})_{\mathbf{\Phi_{ }}\mathbf{\Psi_{ }} \cdots \mathbf{\Gamma_{ }}}$,   
which in turn fixes   
$\displaystyle{\sum_{\gamma} g_{\gamma}({}_{(\alpha)} a_{st})}  (J_{\gamma}|I_{\gamma})_{
\mathbf{\Phi_{ }}\mathbf{\Psi_{ }} \cdots \mathbf{\Gamma_{ }}}$.  
Among those rows with identical sequence $\mathbf{\Phi}$ 
in each $J_{\nu}$ in block standard form,   
the only possible occurrences of $j$ and $j+1$ in its rows are as follows:  
\begin{enumerate}%[ {[[}1{]]} ]  
\item $j$ is followed by $j+1$,  
\item $j$ is followed by an integer larger than $j+1$,  
\item $j$ ends in a row,   
\item $j+1$ is preceded by an integer smaller than $j$, 
\item $j+1$ starts a row. 
\end{enumerate}
Since $J_{\nu}$ is in block standard form, all rows of type $(i)$ 
must occur above all rows of type $(i+1)$ within each block.

After re-numbering the indices $\nu$, 
let $J_1$ have the greatest number of rows, say $M$, of types $(2)$ and $(3)$. 
%If there exists more than one $(J|I)_{\mathbf{\Phi_1}\mathbf{\Phi_2} \cdots \mathbf{\Phi_l}}$ with $M$ rows of types  
%$[[2]]$ and $[[3]]$.  
There may be other Young tableaux, say $J_2,\ldots, J_W$ in 
$(J_{\nu}|I_{\nu})_{\mathbf{\Phi_{ }}\mathbf{\Psi_{ }} \cdots \mathbf{\Gamma_{ }}}$, 
with $M$ rows of types $(2)$ and $(3)$, but we may ignore them because of the uniqueness 
of $(J_{\nu}|I_{\nu})_{
\mathbf{\Phi_{ }}\mathbf{\Psi_{ }} \cdots \mathbf{\Gamma_{ }}}$. 
We label the sequence 
$\mathbf{\Phi_{ }}$ of integers associated to the rows of 
$(J_{\nu}|I_{\nu})_{\mathbf{\Phi_{ }}\mathbf{\Psi_{ }} \cdots \mathbf{\Gamma_{ }}}$  
as 
$\mathbf{\Phi_{}(\nu)}$ $\leq$ $\mathbf{\Psi_{}(\nu)}$ $\leq$ $\ldots$ $\leq$ $\mathbf{\Gamma_{}(\nu)}$. 
Let 
\[ 
\begin{aligned} 
U_{j,j+1} := \{ u_{j,j+1} &= (\I,\ldots, u_{ha_{\phi(1)}},\ldots, u_{ha_{\mu(1)}},\ldots, \I ): 
u_{ha_{\phi(1)}} = \ldots= u_{ha_{\mu(1)}} \mbox{ is the matrix with } \\ 
&\mbox{ 1 along the diagonal, the same variable $u$ in $(j,j+1)$-entry, and 0 elsewhere}  \}. \\ 
\end{aligned}
\] 
Applying $u_{j,j+1}\in U_{j,j+1}$ to $F$, we see that 
$g_1({}_{(\alpha)}a_{st})(J_1|I_1)_{\mathbf{\Phi_{ }}\mathbf{\Psi_{ }} 
\cdots \mathbf{\Gamma_{ }}}$ gives a term 
\[ 
u^M g_1({}_{(\alpha)}a_{st}) (J_1'|I_1)_{\mathbf{\Phi_{ }}\mathbf{\Psi_{ }} \cdots \mathbf{\Gamma_{ }}}, 
\]  
where $J_1'$ is obtained from $J_1$ by replacing each $j$ in rows of type 
$(2)$ and $(3)$ by $j+1$; 
furthermore, $(J_1'|I_1)_{\mathbf{\Phi_{ }}\mathbf{\Psi_{ }} \cdots \mathbf{\Gamma_{ }}}$ is block standard. 
Thus, the tableau $J_1'$ uniquely determines $J_1$ for the following reasons: 
first, all rows of type 
$(3)$ have been changed to rows ending with $j+1$; 
such rows must end with $j+1$ in $J_1$ by our choice of $j$. 
Otherwise, to obtain $J_1$, 
we change $j+1$ to $j$ in $M$ rows of $J_1'$ reading from top to bottom 
(while ignoring those rows which contain both $j$ and $j+1$).

Now, in $u_{j,j+1}.F$, any other occurrence of $(J_1'|I_1)_{\mathbf{\Phi_{}}\mathbf{\Psi_{ }} \cdots \mathbf{\Gamma_{}}}$ is with a coefficient 
$u^k g'({}_{(\alpha)}a_{st})$ 
where $k<M$ and $g'$ is a polynomial in ${}_{(\alpha)}a_{st}$ 
since $j$ was carefully chosen such that $j$ is the biggest integer satisfying $(\dagger)$. 
Since $F$ is a polynomial over a field of characteristic $0$, the coefficient of  
$(J_1'|I_1)_{\mathbf{\Phi_{ }}\mathbf{\Psi_{ }} \cdots \mathbf{\Gamma_{ }}}$   
depends on $u$.   
Thus, $F$ is not fixed by $\mathbb{U}_{\beta}$. 
This shows that $j$ cannot appear in a row of $J_{\nu}$ without $j+1$, which shows that $\mathbb{U}_{\beta}$-invariant polynomials must be of the form \eqref{eq:affine-dk-quiver-general-n-inv-poly}.  
%Thus, if all the terms in a $\mathbb{U}_{\beta}$-invariant polynomial are divisible by $x_{st}$ for some $s$ and $t$, 
%then the polynomial must be of the form \eqref{eq:affine-dk-quiver-general-n-inv-poly}. 
Thus, if some of the terms of a 
$\mathbb{U}_{\beta}$-invariant polynomial $f$ are divisible by $x_{st}$ for some $s$ and $t$, then we write $f$ as 
\[ 
f(x_{st}, {}_{(\alpha)}a_{st}) = g(x_{st}, {}_{(\alpha)}a_{st}) + h({}_{(\alpha)}a_{st}),  
\] where all the terms in $g$ are divisible by $x_{st}$ for some $s$ and $t$ and $h$ is a polynomial in ${}_{(\alpha)}a_{st}$. 
By the first part of this proof, $h\in \mathbb{C}[\mathfrak{t}^{\oplus Q_1}]$ while $g$ must be of the form \eqref{eq:affine-dk-quiver-general-n-inv-poly}. 
It is immediate by the proof of Theorem~\ref{theorem:two-paths-max-quiver-semi-invariants}  
that if none of the terms in a  
$\mathbb{U}_{\beta}$-invariant polynomial are divisible by $x_{st}$ for all $1\leq s\leq n$ and $1\leq t\leq m$, 
then the polynomial is in $\mathbb{C}[\mathfrak{t}^{\oplus Q_1}]$. 
This concludes our proof. 
\end{proof}

\appendix
\bibliography{inv-and-semi-inv-of-all-filtered-quivers}   

\end{document}